\documentclass[10pt]{amsart}
\usepackage{amsmath,amssymb,xcolor}
\usepackage[english]{babel}
\usepackage[usenames,dvipsnames]{pstricks}
\usepackage{epsfig}
\usepackage[all]{xy}
\usepackage{enumitem}

\DeclareMathOperator{\ev}{ev}
\DeclareMathOperator{\TC}{TC}
\DeclareMathOperator{\cat}{cat}
\DeclareMathOperator{\secat}{secat}
\DeclareMathOperator{\nil}{nil}
\DeclareMathOperator{\conn}{conn}
\DeclareMathOperator{\id}{id}
\DeclareMathOperator{\Ker}{Ker}
\DeclareMathOperator{\im}{Im}
\DeclareMathOperator{\Cov}{Cov}

\newcommand{\CC}{\mathord{\mathbb{C}}}
\newcommand{\RR}{\mathord{\mathbb{R}}}
\newcommand{\QQ}{\mathord{\mathbb{Q}}}
\newcommand{\ZZ}{\mathord{\mathbb{Z}}}
\newcommand{\NN}{\mathord{\mathbb{N}}}

\newcommand{\wX}{\widetilde{X}}
\newcommand{\cH}{\check{H}}
\newcommand{\cd}{\mathrm{cd}}

\newcommand{\pullbackcorner}[1][dr]{\save*!/#1-1.5pc/#1:(-1,1)@^{|-}\restore}

\newtheorem{theorem}{Theorem}[section]
\newtheorem{lemma}[theorem]{Lemma}
\newtheorem{proposition}[theorem]{Proposition}
\newtheorem{corollary}[theorem]{Corollary}

\newtheorem{remark}[theorem]{Remark}

\newtheorem{example}[theorem]{Example}

\newcommand{\partsec}[1]{admits a partial section of $\pi_{#1}$}

\title{Topological Complexity of a Map}

\author{Petar Pave\v{s}i\'c}
\address[Petar Pave\v{s}i\'c]{\newline \newline Faculty of Mathematics and Physics,  
\newline University of Ljubljana, 
\newline Jadranska 21, Ljubljana, 
\newline Slovenia}
\email{petar.pavesic@fmf.uni-lj.si}
\thanks{The author was supported by the Slovenian Research Agency research grant P1-0292 and research project J1-7025.}

\begin{document}

\maketitle

\begin{abstract}
We study certain topological problems that are inspired by applications to autonomous robot manipulation.
Consider a continuous map $f\colon X\to Y$, where $f$ can be a kinematic map from the configuration space 
$X$ to the working space $Y$ of a robot arm or a similar mechanism. Then one can associate to $f$ a number
$\TC(f)$, which is, roughly speaking, the minimal number of continuous rules that are necessary 
to construct a complete manipulation algorithm for the device. Examples show that $\TC(f)$ is very sensitive 
to small perturbations of $f$ and that its value depends heavily on the singularities  of $f$. This fact considerably 
complicates the computations, so  we focus here on 
estimates of $\TC(f)$ that can be expressed in terms of homotopy invariants of spaces $X$ and $Y$, or
that are valid if $f$ satisfy some additional assumptions like, for example, being a fibration.

Some of the main results are the derivation of a general upper bound for $\TC(f)$, invariance of $\TC(f)$
with respect to deformations of the domain and codomain, proof that $\TC(f)$ is a FHE-invariant, 
and the description of a cohomological lower bound for $\TC(f)$. Furthermore, if $f$ is a fibration we derive
more precise estimates for $\TC(f)$ in terms of the Lusternik-Schnirelmann category and the topological complexity
of $X$ and $Y$. We also obtain some results for the important special case of covering projections. 
\end{abstract}

\maketitle

\section{Introduction}

In 2003 Michael Farber \cite{Farber:TCMP} introduced the topological complexity of 
a space $X$, denoted $\TC(X)$, as a homotopy-invariant measure of the difficulty to plan 
a continuous motion of a robot in 
the space $X$. Over the years the interest for applications of topological complexity and 
related concepts to problems in
robotics grew into an independent field of research. Topological complexity of a map is 
a natural extension of $\TC(X)$ suggested by Alexander Dranishnikov during the conference 
on Applied Algebraic Topology in Castro Urdiales (Spain, 2014). The new concept opens the 
possibility to model several new concepts in topological robotics. The present author used 
the topological complexity of a map in in \cite{Pav:Complexity} as a measure of manipulation
complexity of a robotic device. That point of view was further developed in 
\cite{Pav:Topologist's view}.
The main thrust of both papers was on applications to kinematic maps that arise in commonly 
used robot configurations. As a consequence, many related theoretical 
question were left aside. The purpose of the present paper is to fill that gap.

Let $f\colon X \rightarrow Y$ be a continuous map: given $x \in X$, $y \in Y$, we look for a path 
$\alpha = \alpha(x,y)$ in $X$ starting at $x$ and ending at a point that is mapped to $y$ by $f$. 
We normally assume 
that $X$ is path-connected and that $f$ is surjective, so that the above problem always has 
a solution.
However, we want the assignment $(x,y) \mapsto \alpha(x,y)$ to satisfy an additional condition, 
namely to be as continuous as possible. More formally, we consider the space $X^I$ of all paths in $X$ 
and the projection map 
$$\pi_f \colon X^I \rightarrow X \times Y \quad \text{where} \quad 
\pi_f(\alpha) = \big(\alpha(0),f(\alpha(1))\big).$$
Then every solution to the above-mentioned problem can be interpreted as a section 
$s\colon X\times Y\to X^I$ to the 
projection $\pi_f$. There are simple examples of maps $f\colon X\to Y$ such that $\pi_f$ 
does not admit a section that
is continuous on entire $X\times Y$. Therefore, one may attempt
to split $X\times Y$ into subspaces, each admitting a continuous section to $\pi_f$.  
The minimal number of elements in such a partition is the \emph{topological complexity} of 
the map $f$.

Topological complexity of a map can be viewed as a natural generalization of the 
topological complexity 
of a single space, introduced by Farber \cite{Farber:TCMP}. However, computation of $\TC(f)$ requires the study 
of a host of new phenomena related to its domain, codomain and singularities. 

In this paper we will not be concerned with the applications of $\TC(f)$ to robotics. Nevertheless to give 
a flavour of the maps which one may want to study, we just mention a variety of situations 
that can be modelled by $\TC(f)$
(see \cite[Section 5]{Pav:Topologist's view} for more details).

\begin{itemize}[leftmargin=*]
\item If $X$ is the configuration space of a system and $f\colon X \rightarrow Y$ is a projection 
to the configuration space of a part or a subsystem, then $\TC(f)$ measures the complexity of manipulation of the components of a 
complex mechanism (e.g a moving platform), where one is only interested in the positioning of some intermediate part
of the structure (e.g. an object on the platform);
\item The complexity of manipulation of a robotic arm is modelled by letting $X$ be a joint space, $Y$ the working space 
and $f: X \rightarrow Y$ the forward kinematic map of the arm (see \cite{Pav:Complexity} for a detailed discussion);
\item Let $X$ be a configuration space of a robotic mechanism where different points of $X$ (i.e. positions of the mechanism) 
are functionally equivalent (e.g. for grasping, pointing,\ldots). If we express functional equivalence in terms of 
the action of some symmetry group $G$, then the manipulation complexity of the device is modelled by 
the topological complexity of the quotient map $X \rightarrow X/G$.
\end{itemize}

We begin the paper with a discussion of the 'correct' definition of the complexity of a map. In fact, 
a straightforward generalization of the standard definition of topological complexity of a space
proposed by  Dranishnikov turned out to be somewhat inadequate for maps with singularities. We
devised an alternative approach which is equivalent to Dranishnikov's when applied to fibrations 
but yields more satisfactory results for general maps. 

The third section is dedicated to a various upper and lower estimates for the
topological complexity of a map. Some of these are valid for arbitrary maps, while other hold for maps
that have some additional properties, e.g. are fibrations or admit a section (see section \ref{sec:summary} for a summary of main results). 

In the final section we specialize to 
maps that are fibrations and express their complexity in terms of other homotopy invariants. This
allows 
computation of topological complexity of many standard fibrations. In particular we show that 
 topological
complexities of covering projections can be viewed as approximations of topological complexity of the
 base space.

\section{Definition of $\TC(f)$}
\label{sec:Definition}

We are going to define the topological complexity of a map in a way that will allow a comparison with
two other related concepts - $\cat(X)$, the \emph{Lusternik-Schnirelmann category} of $X$, and $\TC(X)$,
the \emph{topological complexity} of $X$. In fact all three concepts can be expressed in terms of 
sectional numbers of certain maps. 

Let $p\colon E\to B$ be a continuous surjection. A \emph{section} of $p$ is a right inverse of $p$
i.e.,  a map $s\colon B\to E$, such that $p\circ s=1_B$. Moreover, given a subspace $A\subset B$, a 
\emph{partial section} of $p$ over $A$ is a section of the restriction map $p\colon p^{-1}(A)\to A$. 
If $p$ does not admit a continuous section, it may still happen that it admits sufficiently many  
continuous partial sections so that their domains cover $B$. 

We define  $\sec(p)$, the \emph{sectional number} of $p$  to be the minimal integer $n$ for which there
exists an increasing sequence of open subsets
$$\emptyset = U_0 \subset U_1 \subset U_2 \subset \cdots \subset U_k = B,$$
such that each difference $U_i-U_{i-1}$, $i=1,\ldots,n$ admits a continuous partial section to $p$. 
If there is no such integer $n$, then we let $\sec(p)=\infty$.

A word of warning is in order here, since the above is not the entirely standard definition of sectional number. 
Indeed, sectional number is more commonly  defined as the minimal number of elements in an \emph{open} cover 
of $B$, such that each element admits a continuous partial section to $p$. Let us denote this second quantity 
as $\sec_{\rm op}(p)$. Obviously $\sec(p)\le \sec_{\rm op}(p)$. On the other hand, it is easy to see 
that if $p$ is a fibration and $B$ is an ANR space, then  $\sec(p)$ and $\sec_{\rm op}(p)$ actually 
coincide.
One should also note the similarity between $\sec_{\rm op}(p)$ and ${\rm secat}(p)$,
the \emph{sectional category} of $X$ (also called \emph{Schwarz genus} of $p$, cf. \cite{Schwarz}, \cite{CLOT}). 
The latter counts the minimal number of \emph{homotopy} sections of $p$, therefore $\sec_{\rm op}(p)={\rm secat}(p)$
if $p$ is a fibration, but in general $\sec(p)$ can be much bigger than ${\rm secat}(p)$ (see 
\cite[Section 5]{Pav:Topologist's view} for some specific examples).

We are now ready to state the definition of the Lusternik-Schnirelmann category and the definitions 
of the topological complexity of a space and of a map. For any space $X$ let $X^I$ be the space of 
all continuous 
paths in $X$ (endowed with the compact-open topology) and let $PX$ be the subspace of all based paths 
in $X$ starting at some fixed base-point $x_0\in X$ (which we omit from the notation). 
It is well known that for any point $c\in [0,1]$ the evaluation map 
$$\ev_c\colon X^I\to X,\ \ \ \alpha\mapsto \alpha(c)$$
is a fibration (and similarly for $PX$ in place of $X$, provided that $c\in (0,1]$).  

The \emph{Lusternik-Schnirelmann category} of a space $X$ is defined as 
$$\cat(X)=\sec(\ev_1\colon PX\to X).$$
If $X$ is an ANR, then our definition is equivalent to the standard one that uses open coverings of $X$ by 
categorical subsets.  For the convenience of the reader we list in the next proposition the most important
properties of the Lusternik-Scnirelmann category

\begin{proposition}\ \\[-6mm]
\label{prop:LScat properties}
\begin{enumerate}
\item $\cat(X)=1$ if, and only if $X$ is contractible;
\item Homotopy invariance: $X\simeq Y \Rightarrow \cat(X)=\cat(Y)$;
\item Dimension-connectivity estimate: if $X$ is $d$-dimensional and $(c-1)$-connected, then 
$\cat(X)\le\frac{d}{c}+1$;
\item Cohomological estimate: $\cat(X)\ge \nil\widetilde H^*(X)$, where\\
$\widetilde H^*(X)$ is the ideal of positive-dimensional cohomology classes in $H^*(X)$;
\item Product formula: $\cat(X\times Y)\le \cat(X)+\cat(Y)-1$.
\end{enumerate}
\end{proposition}

More recently M. Farber \cite{Farber:TCMP} introduced the concept of a topological complexity of a space 
in order to provide a crude measure of the complexity of motion planning of mechanical systems, e.g. robot arms.
The \emph{topological complexity} of a (path-connected) space $X$ is
$$\TC(X):=\sec(\pi),\ \ \ \ \text{where}\ \ \pi=(ev_0,\ev_1)\colon X^I\to X\times X.$$
As before, if $X$ is an ANR space, then the above coincides with the Farber's original definition
(cf. \cite{Farber:ITR} or \cite{Pav:FATC}). It is not surprising that many properties of $\TC(X)$
resemble those of $\cat(X)$ and that the two quantities are closely related. The main properties of $\TC(X)$ are 
listed in the following proposition.

\begin{proposition}\ \\[-6mm]
\label{prop:TC properties}
\begin{enumerate}
\item $\TC(X)=1$ if, and only if $X$ is contractible;
\item Homotopy invariance: $X\simeq Y \Rightarrow \TC(X)=\TC(Y)$;
\item Category estimate: $\cat(X)\le TC(X)\le \cat(X\times X)$;
\item If $X$ is a topological group, then $\TC(X)=\cat(X)$;
\item Cohomological estimate: $\cat(X)\ge \nil(\Ker\Delta^*)$, where \\
$\Delta^*\colon\widetilde H^*(X\times X)\to H^*(X)$ is induced by the diagonal $\Delta\colon X\to X\times X$;
\item Product formula: $\TC(X\times Y)\le \TC(X)+\TC(Y)-1$.
\end{enumerate}
\end{proposition}

We may finally turn to the definition of the topological complexity of a map. Let 
$f\colon X\to Y$ be a continuous surjection between path-connected spaces, and let 
$\pi_f\colon X^I\to X\times Y$ be defined as $\pi_f:=(\ev_0,f\circ \ev_1)=(1\times f)\circ \pi$.
Then the \emph{topological complexity} of the map $f$ is defined as
$$\TC(f):=\sec(\pi_f).$$

Clearly $\TC(\id_X) = \TC(X)$, so the topological complexity of a map is a generalization of the 
topological complexity of a single space. We will see later (Example \ref{ex:basicTC}) that 
$\cat(X)=\TC(\ev_1\colon PX\to X)$, so the topological complexity of a map generalizes 
the Lusternik-Schnirelmann category as well.

Most of Section \ref{sec:general estimates} is dedicated to the appropriate 
extensions of Propositions \ref{prop:LScat properties} and \ref{prop:TC properties} for the topological 
complexity of a map.
In the rest of this section we will relate $\TC(f)$ to (partial) sections of $f$, and explain why 
a definition of $\TC(f)$ based on partial sections over open covers of $X\times Y$ does not work well in general.

Let $A \subset X \times Y$, such that $A$ \partsec{f}, say $s\colon A \rightarrow X^I$. 
For a fixed $x_0 \in X$, let 
$\hat{A} = \{y \in Y | (x_0,y) \in A\}$ and define $\hat s\colon \hat{A} \rightarrow X$ by 
$\hat s(y) =s(x_0,y)(1)$. Clearly, $\hat s$ is
a continuous partial section of $f$. Some of the consequences of this follow:
\begin{itemize}[leftmargin=*]
\item If $\pi_f\colon X^I\to X \times Y$ admits a global continuous section, 
then so does $f: X \rightarrow Y$, i.e. $f$ is essentially 
a retraction of $X$ to $Y$. This immediately gives plenty of maps whose complexity is bigger than 1. 
For example, the map 
$f\colon [0,3]\to [0,2]$ given by 
$$f(t):=\left\{\begin{array}{ll}
t & t\in [0,1]\\
1 & t\in [1,2]\\
t-1 & t\in [2,3]
\end{array}\right.$$
(see Figure \ref{fig:TC>1}) clearly does not admit a section, therefore its topological complexity must 
be bigger than one. 
Compare \cite[Section 5]{Pav:Topologist's view} for a general procedure
for constructing maps with contractible domain and codomain and with arbitrarily high topological complexity.

\begin{figure}[ht]
    \centering
    \includegraphics[scale=0.5]{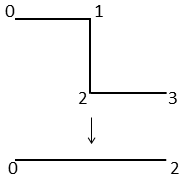}
    \caption{Map whose complexity is bigger than one.}
    \label{fig:TC>1}
\end{figure}

\item If $(x_0,y_0)$ is an interior point of $A$, then the above formula yields a partial section for $f$ 
defined on a neighborhood of
$y_0$. This raises the question of admissible domains for partial sections of $\pi$. In particular, 
if $f$ is not locally sectionable
at some point, then we cannot insist that the domains of partial sections are 
open subsets (as it is otherwise customary in the definition of $\TC(X)$ or $\cat(X)$), because
the topological complexity of such a map would be infinite. On the other hand, we are mostly interested
in the topological complexity of relatively tame maps, whose singular sets are usually closed, so that
our definition based on filtrations of $X\times Y$ by open sets works well 
(see also Section \ref{sec:upper bound} for some general finiteness estimates for $\TC(f)$). 
\end{itemize}

The following alternative description of $\TC(f)$ is often used in applications.

\begin{proposition}
\label{prop:alternative def}
Let $f\colon X\to Y$ be any map. Then $\TC(f)$ equals the minimal integer $n$ such that there exists 
an increasing sequence of closed subsets 
$$\emptyset = C_0 \subseteq C_1 \subseteq C_2 \subseteq \cdots \subseteq C_n = X \times Y$$
where $C_i - C_{i-1}$ \partsec{f} for $i=1,\ldots,n$.

Furthermore, if $X\times Y$ is locally compact, then $\TC(f)$ equals the minimal integer $n$ such 
that there exists a partition of $X\times Y$ into disjoint locally compact subsets $G_1, G_2, \ldots G_n$
where $G_i$ \partsec{f}  for $i=1,\ldots,n$.
\end{proposition}
\begin{proof}
The equivalence of the open and closed definitions follows immediately from De Morgan's Laws and the fact 
that the complement of an open set is a closed set.

As for the second claim, recall that since $X \times Y$ is locally compact, then a subset $G$ is locally compact 
if and only if $G = C_1 - C_2$ for some closed sets $C_1, C_2$. Therefore, given an increasing sequence 
$$\emptyset = C_0 \subseteq C_1 \subseteq \cdots \subseteq C_n = X \times Y$$
where $C_i - C_{i-1}$ \partsec{f}, then the sets $G_i = C_i - C_{i-1}$ are disjoint, locally compact, 
and each $G_i$ \partsec{f}.\\
To prove the converse, take a disjoint partition $X \times Y=G_1 \sqcup G_2 \sqcup \cdots \sqcup G_n$, where 
$G_i$ are locally compact and admit a partial section to $\pi_f$ and each $G_i$ as a difference $G_i = A_i - B_i$ 
of two closed sets. We can then define the following increasing sequence of closed sets:
$$C_1 = \bigcup_{i=1}^n B_i\ \ \ \text{and}\ \ \  C_i:= C_{i-1} \cup A_{i-1} \text{ for } i=2,\ldots, n.$$
Note that $C_1$ can also be expressed as
$C_1=\bigcup_{i=1}^n \left( \bigcup_{j=1}^n  G_i \cap B_j \right).$
Since $\bigcup_{j=1}^n G_i \cap B_j \subset G_i \subset A_i$, we see that the sets $\bigcup_{j=1}^n (G_i \cap B_j)$ are separated 
from one another and so $C_1$ \partsec{f}.\\
Furthermore, since $C_i - C_{i-1} \subset A_{i-1} - B_{i-1},$ we conclude that $C_i - C_{i-1}$ admit a partial section to $\pi_f$ 
for $i=1,\ldots,n$
\end{proof}

\begin{remark}
\label{rem:TC def}
Srinivasan \cite{Srinivasan} has recently proved that for $X$ a compact metric ANR one can equivalently define 
$\cat(X)$ 
by  partitioning $X$ into arbitrary categorical subsets. The proof is based on extensions 
of maps from a subset of $X$ to a suitably constructed open neighbourhood 
(cf. \cite[Corollary 2.8]{Srinivasan}).
Her approach can be extended to the case of topological complexity of a space, but the above examples 
show that even for very simple maps the choice of the domains for partial sections can greatly affect 
the outcome. We will return to this question in Section \ref{sec:Topological complexity of a fibration}.
\end{remark}

\section{Estimates of $\TC(f)$ for arbitrary maps}
\label{sec:general estimates}

From this point on we will assume that all spaces under consideration are metric absolute neighbourhood 
retracts (metric ANR's). As explained before, this will allow a direct comparison between the $\TC(f)$ and 
the category or topological complexity of its domain and codomain.
The following simple lemma will be particularly useful for the comparison of the topological complexity of related maps. 

\begin{lemma}
\label{Comparison_Lemma}
Let $f: X \rightarrow Y$ and $f': X' \rightarrow Y'$ be any maps, and suppose 
there exists a map $h\colon Y \to Y'$ with the following property: whenever $f'$ admits a partial section over some
$A \subseteq Y'$, $f$ admits a partial section over $h^{-1}(A)$ as depicted in
the following diagram:
$$\xymatrix{
& &X\ar[d]_f & & X' \ar[d]^{f'}\\
& &Y\ar[rr]^h & & Y'\\
h^{-1}(A) \ar@{^(->}[urr]\ar@{-->}[uurr]\ar[rr]^-h & & A \ar@{^(->}[urr]\ar[uurr]}
$$
Then $\sec(f) \leq \sec(f')$.
\end{lemma}
\begin{proof}
Suppose that $\sec(f') = k$ and that 
$$U_0 \subset U_1 \subset \cdots \subset U_k = X' \times Y'$$
is an increasing sequence of open subsets where $f'$ admits a partial section over $U_i - U_{i-1}$ for every $1 \leq i \leq k$. Then $f$ admits a partial section over each $h^{-1}(U_i - U_{i-1})$  by hypothesis. 
Since $h$ is continuous, all $h^{-1}(U_i)$ are open and so
$$\emptyset = h^{-1}(U_0) \subset h^{-1}(U_1) \subset \cdots \subset h^{-1}(U_k) = Y$$
is an increasing sequence of open subsets where for every 
$1 \leq i \leq k$ the restriction of $f$ admits a continuous section over $h^{-1}(U_i) - h^{-1}(U_{i-1}) = h^{-1}(U_i - U_{i-1})$. We conclude that $\sec(f) \leq k = \sec(f')$.
\end{proof}

\begin{proposition}
\label{prop:TC and cat}
For any map $f: X \rightarrow Y$, we have 
$$\TC(f) \geq \cat(Y).$$
\end{proposition}
\begin{proof}
Fix $x_0\in X$ and consider the inclusion $h\colon Y\hookrightarrow X\times Y$, given as $h(y):=(x_0,y)$. If $A\subseteq X\times Y$
admits a partial section $\sigma\colon A\to X^I$ to $\pi_f$, then one can easily check that 
$$\ev_1\circ f_*\circ\sigma\circ h =1_{h^{-1}(A)}, $$
where $f_*\colon P_{x_0}X\to P_{f(y_0)}Y$ denotes the post-composition by $f$. 
Therefore 
$$f_*\circ\sigma\circ h\colon h^{-1}(A)\to Y^I$$ 
is a partial section to the map $\ev_1\colon PY\to Y$ over $h^{-1}(A)$.
By Lemma \ref{Comparison_Lemma} we conclude that 
$$\TC(f)=\sec(\pi_f)\ge\sec(\ev_1)=\cat(Y).$$
\end{proof}

Another lower bound for $\TC(f)$ is given by the number of partial continuous 
sections of $f$. 

\begin{proposition}
\label{prop:TC and sec}
For any map $f: X \rightarrow Y$, we have 
$$\TC(f) \geq \sec(f).$$
In particular, if $\TC(f)=1$, then $f$ admits a continuous section.
\end{proposition}
\begin{proof}
Fix $x_0\in X$ and define $h\colon Y\to X\times Y$ as in the previous proof. 
If $\sigma\colon A\to Y^I$ is a partial section to $\pi_f$ then 
$$f\circ\ev_1\circ \sigma \circ h=1_{h^{-1}(A)},$$
therefore  $\ev_1\circ\sigma\circ h\colon h^{-1}(A)\to X$ is a partial section
to $f$. By Lemma \ref{Comparison_Lemma} $\TC(f)=\sec(\pi_f)\ge\sec(f)$.
\end{proof}

Observe that if $f$ is a fibration, then $\sec(f)\le\cat(Y)$, because $f$ admits a partial section over every categorical subset of $Y$. Therefore, for fibrations
Proposition \ref{prop:TC and cat} implies Proposition \ref{prop:TC and sec}.

Before proceeding let us introduce the following notation.
Given a homotopy $H: X \times I \rightarrow Y$, we can use adjunction
to define continuous functions 
$\overrightarrow{H}, \overleftarrow{H}\colon X\to Y^I,$
by the formulas  
$$\overrightarrow{H}(x)(t) := H(x,t)\ \ \text{and}\ \ 
\overleftarrow{H}(x)(t) := H(x,1-t).$$

\begin{proposition}
\label{prop:TC and catX}
If there exists $y_0\in Y$ such that the fibre $f^{-1}(y_0)$ of the map 
$f: X \rightarrow Y$ is categorical in $X$, then 
$$\TC(f) \geq \cat(X).$$
\end{proposition}
\begin{proof}
Define $h\colon X\to X\times Y$ by $h(x):=(x,y_0)$. 
By assumption, there exists a homotopy $H\colon f^{-1}(y_0)\times I\to X$
which deforms $f^{-1}(y_0)$ to a point. 
If $\sigma\colon A\to X^I$ is a partial section to $\pi_f$, then
it is easy to verify that the map 
$$\overrightarrow{H}\circ \ev_1\circ \sigma\circ h$$
determines a deformation of $h^{-1}(A)\subseteq X$ to a point in $X$. 
As before, by Lemma \ref{Comparison_Lemma} we conclude that $\TC(f)\ge\cat(X)$. 
\end{proof}

\subsection{Effect of pre-composition on the complexity}

Our next objective is to study the effect that pre-composition by
a map has on the complexity of $f$.

\begin{theorem}
\label{thm:precomp}
Consider the diagram $\widehat{X} \xrightarrow{v} X \xrightarrow{f} Y$.
\begin{enumerate}
\item[\textbf{a)}] If $v$ admits a right homotopy inverse (i.e., a map 
$u\colon Y\to X$, such that $vu \simeq 1$), then $\TC(fv) \geq \TC(f)$
\item[\textbf{b)}] If $v$ admits a left homotopy inverse (a map $u$ 
such that $uv \simeq 1$) and if $fvu = f$, then $\TC(fv) \leq \TC(f)$.
\item[\textbf{c)}] If $v$ admits a left homotopy inverse $u$,  if $fvu \simeq f$ and if additionally $fv$ is a fibration, then $\TC(fv) \leq \TC(f)$
\end{enumerate}
\end{theorem}
\begin{proof}
\begin{enumerate}
\item[\textbf{a)}] Suppose $A \subset \widehat X \times Y$ \partsec{fv}, say $\alpha_{fv}\colon  A \rightarrow \widehat X^I$ and $H\colon  vu \simeq 1$. Then the formula
$$\alpha_f(y,z) := \overleftarrow{H}(y) \cdot (v \circ \alpha_{fv}(u(y),z))$$
defines a continuous partial section on $(u \times 1)^{-1}(A)$. Since $(u\times 1)\colon X \times Y \rightarrow \widehat X \times Y$ is continuous, then  $\TC(f) \leq \TC(fv)$ by Lemma \ref{Comparison_Lemma}.
\item[\textbf{b)}] Suppose $A \subset X \times Y$ \partsec{f}, say $\alpha_{f} : A \rightarrow X^I$. Let $H : uv \simeq 1$. Then the formula
$$\alpha_{fv}(x,z) := \overleftarrow{H}(x) \cdot (u \circ \alpha_f(v(x),z))$$
defines a continuous map on $(v \times 1)^{-1}(A)$. Observe that $\alpha_{fv}(x,z)$ is path starting at $x$ and ending at $u(y')$ where $f(y')=z$. Thus $fv \circ \alpha_{fv}(x,z)$ ends at $fv(u(y'))=f(y')=z$. Therefore $\alpha_{fv}$ is a continuous partial section for $\pi_{fv}$. Again, $(v \times 1): \widehat X \times Y \rightarrow X \times Y$ is continuous, so, by \ref{Comparison_Lemma}, $\TC(fv) \leq \TC(f)$.
\item[\textbf{c)}] Suppose $A \subset X \times Y$ \partsec{f}, say $\alpha_{f} : A \rightarrow X^I$. Let $H : uv \simeq 1$ and $K : fvu \simeq f$. Let $\Gamma_{fv} : \widehat X \sqcap Y^I\to X^I$ denote the lifting function for the fibration $fv$. Then the formula
$$\alpha_{fv}(x,z) := \overleftarrow{H}(x) \cdot (u \circ \alpha_f (v(x), z)) \cdot \Gamma_{fv}(u(x'),\overrightarrow{K}(x')),$$
where $x' = \alpha_f(v(x),z)(1)$, defines a continuous partial section for $(v \times 1)^{-1}(A)$. Thus by \ref{Comparison_Lemma} $\TC(fv) \leq \TC(f)$.
\end{enumerate}
\end{proof}

Furthermore, we have the following surprising result  that the complexity of a map cannot increase 
if we pre-compose it with a fibration.

\begin{theorem}
\label{thm:precomp fibration}
If $v\colon X\to Y$ is a fibration, then $\TC(fv)\le\TC(f)$ for every $f\colon X\to Y$.
\end{theorem}
\begin{proof}
Let $\alpha_f\colon A\to Y^I$ be a partial section for $\pi_f\colon PY\to Y\times Z$ over some $A\subseteq Y\times Z$. 
Then the formula 
$$\alpha_{fv}(x,y):=\Gamma_v(x,\alpha_f(v(x),z))$$
defines a partial section for $\pi_{fv}$ over $(v\times 1)^{-1}(A)$. As usual, this implies that $\TC(fv)\le\TC(f)$, 
\end{proof}

The above theorems have several interesting corollaries. First, we deduce the following important invariance property, 
which states that the complexity of the map is not altered 
by a deformation retraction of the domain. 

\begin{corollary}
\label{cor:DR}
If $v\colon \widehat{X}\to X$ is a deformation retraction, then for every 
$f\colon X\to Y$ we have
$\TC(f)=\TC(fv)$.
\end{corollary}
\begin{proof}
Let $i\colon X\hookrightarrow \widehat{X}$ be the inclusion, so that $vi=1_X$ 
and $iv\simeq 1_{\widehat X}$. Then Theorem \ref{thm:precomp}(a) implies 
that $\TC(fv)\ge\TC(f)$, while statement (b) and the observation
that $fhi=f$ gives $TC(fh)\le\TC(f)$.
\end{proof}

It is important to keep in mind that the deformation retraction in the statement of the above Corollary cannot be replaced by an arbitrary homotopy equivalence. 
For example, the identity map $1_{[0,2]}$  and the map $f$ depicted in Figure \ref{fig:TC>1} have homotopy equivalent 
domains, and yet the complexity of $f$ is $\TC(f)=2$, while $\TC(1_{[0,2]})=1$. The problem is that a homotopy equivalence $u$
between the domains cannot be chosen so to be a fibrewise map over the base $[0,2]$, i.e. so that the following 
diagram strictly commutes: 
$$\xymatrix{
[0,3] \ar[rr]^u\ar[dr]_f & & [0,2] \ar[dl]^{1_{[0,2]}} \\
& [0,2]
}$$

Nevertheless, if $f\colon X\to Y$ is a homotopy equivalence, then Corollary \ref{cor:fib & sec}(b) bellow applies so we have $\TC(f)\ge\TC(X)=\TC(Y)$.

\begin{corollary}
\label{cor:fib & sec}
\begin{enumerate}
\item[\textbf{a)}] 
If $f\colon X\to Y$ is a fibration, then $\TC(f)\le TC(Y)$.
\item[\textbf{b)}]
If $f\colon X\to Y$ admits a homotopy section, then $\TC(f)\ge TC(Y)$.
\item[\textbf{c)}]
If $f\colon X\to Y$ is a fibration that admits a homotopy section, 
then $\TC(f)=TC(Y)$.
\end{enumerate}
\end{corollary}
\begin{proof}
Consider the following diagram:
$$\xymatrix{
X \ar[r]^f & Y \ar@{=}[r]^{1_Y} & Y
}$$
If $f$ is a fibration, then by Theorem \ref{thm:precomp fibration}
$$\TC(f)=\TC(1_Y\circ f)\le \TC(1_Y)=\TC(Y).$$
On the other hand, if $f$ admits a homotopy section $s\colon Y\to X$, then by 
Theorem \ref{thm:precomp}(a) 
$$\TC(f)=\TC(1_Y\circ f)\ge \TC(1_Y)=\TC(Y).$$
By putting together (a) and (b) we get (c).
\end{proof}

\subsection{Invariance with respect to homotopy}

Recall that two maps $f\colon X\to Y$ and $f'\colon X'\to Y$ are said to be \emph{fibre homotopy equivalent} (or FHE-equivalent) if there is a commutative diagram diagram of the form 
$$\xymatrix{
X \ar@<0.7ex>[rr]^u \ar[dr]_f & & X' \ar@<0.7ex>[dl]^{f'} \ar[ll]^v \\
& Y
}$$
and the maps $u\circ v$ and $v\circ u$ are homotopic to the respective identity
map by fibre-preserving homotopies. It is not surprising that topological complexities of fibre-homotopic maps are equal. In fact, a little more is true:

\begin{corollary}
\label{cor:FHE}
Given $f\colon X\to Y$ and $g\colon X'\to Y$ assume that there exist fibrewise maps $u\colon X\to X'$ and $v\colon X'\to X$ 
that homotopy inverses one to the other. Then $\TC(f)=\TC(f')$.

In particular, the topological complexity is a FHE-invariant.
\end{corollary}
\begin{proof}
By Theorem \ref{thm:precomp}(a) we have
$$\TC(f)=\TC(f'u)\ge \TC(f')=\TC(fv)\ge\TC(f),$$
therefore $\TC(f)=\TC(f')$.
\end{proof}

The following proposition shows that the fibrations have minimal complexity
within their homotopy class.

\begin{proposition}
If $f \simeq g : X \rightarrow Y$ and $f$ is a fibration, then $\TC(f) \leq \TC(g)$.
\end{proposition}
\begin{proof}
Let $H: f \simeq g$, and let  $\Gamma : X \sqcap Y^I \rightarrow X^I$ denote the lifting function for the fibration $f$.\\
Suppose $A \subset X \times Y$ \partsec{g}, say $\alpha: A \rightarrow X^I$. Then for every $(x,y) \in A$, $\alpha(x,y)$ is a path in $X$ starting at $x$ and ending at $x'$ such that $g(x')=y$. Observe that $x' = \ev_1(\alpha(x,y))$ is continuously dependent on $(x,y)$.\\
Define $\bar{\alpha}(x,y) := \alpha(x,y) \cdot \Gamma(x',\overrightarrow{H}(x'))$. Clearly, $\bar{\alpha}$ is a continuous section of $(1 \times f) \circ \ev_{0,1}$. Thus by  \ref{Comparison_Lemma}, $\TC(f) \leq \TC(g)$.
\end{proof}

In particular, we have

\begin{corollary}
If $f,g: X \rightarrow Y$ are homotopic fibrations, then $\TC(f) = \TC(g)$.
\end{corollary}

Another important consequence of Theorem \ref{thm:precomp} is that the complexity cannot increase if we replace a map by a fibration.

\begin{corollary}
\label{cor:fibsubst}
If $\bar{f}: \bar{X} \rightarrow Y$ is the fibrational substitute for $f : X \rightarrow Y$, then $\TC(\bar{f}) \leq \TC(f)$. 
Equality holds if $f$ is a fibration.
\end{corollary}
\begin{proof}
Since $\bar{f}$ is the fibrational substitute for $f$, we have the following diagram
$$\xymatrix{
X \ar@{^(->}@<0.5ex>[rr]^i\ar[dr]_f & & {\overline{X}} \ar@<0.5ex>[ll]^h\ar[dl]^{\bar{f}} \\
& Y
}$$
where $h$ is a fibration, $vu=1_X$ and $uv\simeq 1_{\overline X}$. Then the first claim follows by Theorem \ref{thm:precomp}(a) 
because $\TC(\overline{f})=\TC(fh)\le\TC(f)$. Moreover, if $f$ is a fibration, then so is $fh$, hence $\TC(fh)\ge\TC(f)$
Theorem \ref{thm:precomp}(c). 
\end{proof}

\subsection{Effect of post-composition on the complexity}

Next we study the effect that the post-composition by a map has on 
the topological complexity.

\begin{proposition}
\label{prop:postcomp}
Consider the diagram $X \xrightarrow{f} Y \xrightarrow{v} \widehat Y$.
\begin{enumerate}
\item[\textbf{a)}] If $v$ admits a right inverse (section) $u\colon \widehat Y\to Y$, then $\TC(f) \geq \TC(vf)$
\item[\textbf{b)}] If $v$ admits a left homotopy inverse $u\colon \widehat Y\to Y$ and if $f$ is a fibration, then
$\TC(f)\le \TC(vf)$.
\end{enumerate}
\end{proposition}
\begin{proof}
\begin{enumerate}
\item[\textbf{a)}]
Let $\pi_f\colon X^I\to X\times Y$ admit a partial section  $\alpha \colon A\to X^I$  for some $A\subseteq X\times Y$.
Then the formula 
$$\alpha_{vf}(x,z):=\alpha_f(x,(u(z))$$
defines a path starting at $x$ and ending at some $x'$, such that $f(x')=u(z)$, therefore $vf(x')=vu(z)=z$.
It follows that $\alpha_{vf}$ defines a partial section for $\pi_{vf}$ over $(1\times u)^{-1}(A)\subseteq X\times \widehat Y$.
As before, this implies $\TC(f)\ge\TC(vf)$.
\item[\textbf{b)}]
Let $H\colon Y\times I\to Y$ be the homotopy from $uv$ to $1_Y$, and let $\alpha_{vf}\colon A\to X^IX$ be a partial
section for $\pi_{vf}$ for some $A\subseteq X\times \widehat Y$. Then for every $(x,y)\in (1\times v)^{-1}(A)$ the formula
$\alpha_{fv}(x,v(y))$ gives a path in $X$ starting at $x$ and ending at some $x'$, such that $vf(x')=v(y)$.  
Consequently $uvf(x')=uv(y)$, so $\overleftarrow{H}(f(x'))\cdot \overrightarrow{H}(y)$ is a path in $Y$ starting at
$f(x')$ and ending at $y$. Therefore, the formula
$$\alpha_f(x,y):=\alpha_{vf}(x,v(y))\cdot\Gamma_f(x',\overleftarrow{H}(f(x'))\cdot \overrightarrow{H}(y))$$
defines a partial section to $\pi_f$ over $(1\times v)^{-1}(A)$. Again, we conclude that $\TC(vf)\ge\TC(f)$.
\end{enumerate}
\end{proof}

The following result complements Corollary \ref{cor:fib & sec}(b):

\begin{corollary}
\label{cor:sec and TCX}
If $f\colon X\to Y$ admits a section, then $\TC(f)\le\TC(X)$.
\end{corollary}
\begin{proof}
Let $i\colon Y\to X$ be a right inverse for $f$ and apply Proposition \ref{prop:postcomp} (a) to the diagram 
$$\xymatrix{
X \ar@{=}[r]  & X \ar@<0.5ex>[r]^f & Y \ar@<0.5ex>[l]^i
}$$
Then we have
$$\TC(f)=\TC(f\circ 1_X)\le \TC(1_X)=\TC(X)\; .$$
\end{proof}

Observe, that the last result together with Corollary \ref{cor:fib & sec} yield the following very useful estimate: 
if $f\colon X\to Y$ admits a section, then 
$$\TC(X)\ge\TC(f)\ge\TC(Y).$$

The next result is analogous to Corollary \ref{cor:DR}, but it requires $f$ to be a fibration. 

\begin{corollary}
If $v\colon Y\to\widehat Y$ is a deformation retraction then $\TC(vf)=\TC(f)$ for every fibration $f\colon X\to Y$. 
\end{corollary}
\begin{proof}
By assumption, there is a map $i\colon \widehat{Y}\to Y$ such that $iv=1_{\widehat{Y}}$ and $vi\simeq 1_Y$. 
Then part (a) of Proposition \ref{prop:postcomp} implies that $\TC(vf)\ge\TC(f)$, while part (b) implies 
$\TC(vf)\le\TC(f)$.
\end{proof}

In other words, if $f$ is a fibration, one cannot alter its complexity by deforming its codomain. 
This no longer needs to be true if $f$ is not a fibration. As an easy example, let 
$f\colon [0,3]\to [0,2]$ be the map consider before, and let $v\colon [0,2]\to [0,1]$ be given 
as 
$$v(t):=\left\{\begin{array}{ll}
t & t\in [0,1]\\
1 & t\in [1,2] 
\end{array}\right.$$
Clearly, $v$ is a deformation retraction and $\TC(vf)=1$, while $\TC(f)=2$. 

It is well-known (and easy to prove) that $\TC(X)=1$ if, and only if, $X$ is
contractible. An analogous characterization of maps whose complexity is 
equal to 1 is more elusive.

\begin{proposition}
The following statements are equivalent for a map $f\colon X\to Y$:
\begin{enumerate}
\item $\TC(f)=1$ and at least one fibre of $f$ is categorical in $X$.
\item $X$ is contractible and $f$ admits a continuous section.
\end{enumerate}
\end{proposition}
\begin{proof}
Assume 1.: then by Proposition \ref{prop:TC and sec} $f$ admits a continuous 
section, and by  Proposition \ref{prop:TC and catX} $\cat(X)=1$, therefore
$X$ is contractible. 

Conversely, if we assume 2., then Corollary \ref{cor:sec and TCX} implies
$\TC(f)\le\TC(X)=1$, therefore $\TC(f)=1$.
\end{proof}

However, note that if $Y$ is contractible then Corollary \ref{cor:fib & sec} b)
implies that the complexity of the projection $\mathrm{pr}\colon Y\times F\to Y$
is equal to 1 regardless of the fibre $F$. 

\subsection{A general upper bound for $\TC(f)$}
\label{sec:upper bound}

All upper estimates for $\TC(f)$ that we considered so far required quite 
restrictive assumptions on the map $f$ like being a fibration or admitting
a (homotopy) section. The following theorem gives an upper estimate of $\TC(f)$
for general $f$.

Recall that subspaces $A,B$ of a topological space are said to be \emph{separated} if $\overline A\cap B=A\cap \overline B=\emptyset$. 
It is easy to verify that a function defined on $A\cup B$ is continuous if, and
only if, its restrictions to $A$ and $B$ are continuous.

\begin{theorem}
\label{thm:upper bound}
Topological complexity of a map $f\colon X\to Y$ is bounded above by
$$\TC(f)\le \cat(X)+\cat(X)\cdot\sec(f)-1.$$
\end{theorem}
\begin{proof}
Let $\cat(X)=n$, so that there is an open filtration of $X$
$$\emptyset=U_0\le U_1\le\ldots\le U_n=X,$$
such that for each $i$ the difference $U_i-U_{i-1}$
is categorical in $X$, i.e., there exists a homotopy 
$H_i\colon I\to X$ between the inclusion $U_i-U_{i-1}\hookrightarrow X$ and the constant
map to $x_0\in X$. 

If $V\subset Y$ admits a partial section $s\colon V\to Y$, then $V$ can be split into $n$
subsets $V\cap s^{-1}(U_i-U_{i-1})$, $i=1,\ldots,n$, such that for each $i$ there is a homotopy
$K_i\colon \big(V\cap s^{-1}(U_i-U_{i-1}\big)\times I\to X$ from the restriction of $s$ to 
the constant map to $x_0$. As a consequence, there is an open filtration of $Y$
$$\emptyset=V_0\le V_1\le\ldots\le V_m=Y$$
where $m=\cat(X)\cdot\sec(f)$, such that on each difference $V_j-V_{j-1}$ there exists 
a homotopy $K_j$ between a section $s_j\colon V_j-V_{j-1}\to X$ to $f$ and the constant map.

The formula
$$\sigma_{i,j}(x,y):=\overrightarrow{H}_i(x)\cdot\overleftarrow{K}_j(y)$$
clearly defines a partial section to $\pi_f$ over $(U_i-U_{i-1})\times (V_j-V_{j-1})$. 

For every $2\le k\le m+n$ let $W_k:= \bigcup_{i+j\le k} U_i\times V_j$. 
Then 
$$W_2\subset W_2\subset\cdots W_{m+n}=X\times Y$$
 is an open filtration (of length $m+n-1$) of $X\times Y$ and for each $k$
$$W_k-W_{k-1}= \bigcup_{i+j= k} (U_i-U_{i-1})\times (V_j-V_{j-1}).$$
Observe that the sets in the above union are separated, which implies that 
partial section $\sigma_{i,j}$ for $i+j=k$ define a continuous partial section
on $W_k-W_{k-1}$.  
We have thus proved that $\TC(f)\le \cat(X)+\cat(X)\cdot\sec(f)-1.$
\end{proof}

The exact value of $\sec(f)$ is often hard to compute, so we mostly rely on 
the following coarser but easily computable estimate.

\begin{corollary}
\label{cor:upper bound}
Assume that the map $f\colon X\to Y$ is simplicial with respect to some choice of
triangulations on $X$ and $Y$. Then 
$$\TC(f)\le \cat(X)\cdot (\dim(Y)+1)-1.$$
\end{corollary}
\begin{proof}
It is sufficient to prove that under the assumptions $\sec(f)\le\dim(Y)$. 
To this end let $K$ and $L$ be simplicial complexes that triangulate respectively $X\approx |K|$ and $Y\approx |L|$, 
and with respect to which the map $f$ is 
simplicial. Consider the filtration of $Y$ by subcomplexes
$$|L^{(0)}|\subset |L^{(1)}|\subset \cdots |L^{(\dim Y)}|=Y$$
and observe that for every $1\le i\le \dim Y$ the difference
$|L^{(i)}|-|L^{(i-1)}|$ is a separated union of open $i$-simplices. 
Since the map $f$ is simplicial, it clearly admits a continuous section over each
open $i$-simplex, and thus a  continuous section over their separated union 
$|L^{(i)}|-|L^{(i-1)}|$. This shows that $\sec(f)\le\dim(Y)$, which together 
with Theorem  \ref{thm:upper bound} implies our claim.
\end{proof}

\subsection{Cohomological estimate of $\TC(f)$}

We mentioned in the Introduction the cohomological lower bound for topological complexity of 
a space
$$\TC(X) \geq \nil ( ( \ker \Delta^*\colon H^*(X \times X) \rightarrow H^*(X)),$$
which is widely used in computations of topological complexity.
Here $\Ker\Delta^*$ is the ideal of 'zero divisors' (cf. \cite{Farber:TCMP})
and its nilpotency $\nil(\Ker\Delta^*)$ is the minimal  integer $n$ for which 
every product of $n$ elements in $\Ker\Delta^*$ is equal to zero. 
We will present a similar estimate for the topological complexity of a map
(a variant of which was already used in \cite{Pav:Complexity}).

To formulate our results in full generality we need a form of the relative cohomology product 
$$H^*(X,A)\times H^*(X,B)\to H^*(X,A\cup B).$$ 
which holds for subspaces $A,B\subset X$ that are not necessarily excisive, 
as required in the case of singular cohomology (cf. \cite[Definition VII, 8.1]{Dold}).
For completeness we state and prove the relevant result.

We will follow the notation and definitions of \cite[Sections IV,8 and VIII,6]{Dold}.
Given an Euclidean Neighbourhood Retract (ENR)  $E$, a  \emph{triad} in $E$ consists of 
subspaces $A,B\subseteq X$ of $E$. We write pairs $(X,A)$ and $(X,B)$ instead of 
$(X;A,\emptyset)$ and $(X;\emptyset,B)$. Triads are naturally ordered by inclusion:
$(X;A,B)\subseteq (X';A',B')$ if $X\subseteq X', A\subseteq A', B\subseteq B'$. 
A triad $(X;A,B)$ is \emph{locally compact/open}, if $X,A,B$ are locally compact/open in $E$.

\begin{proposition}
\label{prop:Cech rel prod}
Let $(X;A,B)$ be a locally compact triad in an ENR space $E$. Then one can define a relative 
cohomology product in \v Cech cohomology:
$$\cH^*(X,A)\times\cH^*(X,B)\to \cH^*(X,A\cup B)$$
which is associative, unital and graded-commutative, and for every map of triads 
$f\colon (X';A',B')\to (X;A,B)$ satisfies the naturality property 
$$f^*(u\cdot v)=f^*(u)\cdot f^*(v), \ \ u\in\cH^*(X,A), v\in\cH^*(X,B).$$
\end{proposition}
\begin{proof}
For a locally compact pair $(X,A)$ in $E$ let $\Lambda(X,A)$ be the set of all open pairs
$(U,V)$ in $E$ containing $(X,A)$, ordered downward by inclusion. Singular cohomology groups
$H^*(U,V)$ (we omit coefficients from the notation) form a direct system of abelian groups
indexed by $\Lambda(X,A)$, and we thus obtain \emph{\v Cech cohomology} groups of $(X,A)$ as
a direct limit of singular cohomology groups (cf. \cite[Definition 6.1]{Dold})
$$\cH^*(X,A)=\varinjlim \big\{H^*(U,V)\mid (U,V)\in\Lambda(X,A)\big\}.$$   
Let $\Lambda(X;A,B)$ denote the set of all open triads in $E$ containing $(X;A,B)$ and,
ordered by inclusion. 
For every inclusion $(U;V,W)\subseteq (U';V',W')$ we have a commutative diagram
$$\xymatrix{
H^*(U',V')\times H^*(U',W') \ar[rr]^-{\smile} \ar[d] & & H^*(U',V'\cup W')\ar[d]\\
H^*(U,V)\times H^*(U,W) \ar[rr]^-{\smile} & & H^*(U,V\cup W)
}$$
where the vertical maps are induced by inclusion and the horizontal maps are given by
the usual product in singular cohomology. Observe that the obvious projections 
$\Lambda(X;A,B)\to\Lambda(X,A)$ and 
$\Lambda(X;A,B)\to\Lambda(X,B)$ are cofinal morphisms of direct systems, and that the same
holds for the function $\Lambda(X;A,B)\to\Lambda(X,A\cup B)$ given as $(U;V,W)\mapsto 
(U,V\cup W)$. 
Therefore, we may pass to the respective direct limits in the above diagram 
and obtain the relative cohomology
product in \v Cech cohomology. Its properties clearly follow from the analogous properties
of the product in singular cohomology. 
\end{proof}

Let us denote $x|_A:=i^*(x)\in \cH^*(A)$, where $i$ is the inclusion of 
$A$ in $X$ and $x\in\cH^*(X)$. 

\begin{corollary}
\label{cor:product}
Let $(X;A,B)$ be locally compact triad in some ENR. If $x,y\in\cH^*(X)$ satisfy 
$x|_A=0$ and $y|_B=0$, then $(x\cdot y)|_{A\cup B}=0$.
\end{corollary}
\begin{proof}
Let us apply the above Proposition to the inclusion of triads 
$j\colon (X;\emptyset,\emptyset)\hookrightarrow (X;A,B)$. If $x|_A=0$ then exactness of 
the cohomology sequence of a pair implies that $x=j^*(\overline x)$ for some 
$\overline x\in H^*(X,A)$. Similarly, $y=j^*(\overline y)$ for some $\overline y\in H^*(X,B)$. 
Then by the naturality part of  Proposition \ref{prop:Cech rel prod} we have 
$$x\cdot y=j^*(\overline x)\cdot j^*(\overline y)=j^*(\overline x\cdot\overline y)$$
so by exactness $(x\cdot y)|_{A\cup B}=0,$
\end{proof}

Let $\sigma\colon A\to X^I$ be a partial section to $\pi_f\colon X^I\to X\times Y$ and consider the following diagram:
$$\xymatrix{
 & X^I\ar[d]^{\pi_f} \ar[r]^{\ev_1\simeq} &X\ar[dl]^-{(1,f)}\\
 A \ar[ru]^\alpha \ar@{^(->}[r]_-i & X\times Y
}$$
in which the right-hand triangle is homotopy commutative. 
By applying the \v Cech cohomology functor (with any ring coefficients) and identifying 
$\cH^*(X^I)$ with $\cH^*(X)$ we obtain a commutative diagram
$$\xymatrix{
    & \cH^*(X) \ar[dl]_-{\alpha^*} \\
 \cH^*(A) & \cH^*(X\times Y) \ar[l]^-{\ \ i^*} \ar[u]_{(1,f)^*}
}$$
Clearly, for every class $u\in \Ker (1,f)^*$ we have $u|_A=0$.  
If $\TC(f)=n$, then by Proposition \ref{prop:alternative def} 
there is a covering of $X\times Y$ by locally compact subsets $A_1,\ldots,A_n$, 
such that each $A_i$ admits a partial section to $\pi_f$.
Then for $u_1,\ldots,u_n\in \Ker(1,f)^*$ we may inductively apply 
Corollary \ref{cor:product} to obtain 
$$u_1\cdot\ldots\cdot u_n=(u_1\cdot\ldots\cdot u_n)|_{A_1\cup\ldots\cup A_n}=0.$$

\begin{theorem}
\label{thm:coho bound}
For every map $f\colon X\to Y$ between locally compact subspaces of some ENR 
we have the estimate
$$\TC(f)\ge\nil\big(\Ker(1,f)^*\colon \cH^*(X\times Y)\to \cH^*(X)\big),$$
where $\cH^*$ denotes \v Cech cohomology with any ring coefficients.

Moreover, if both $X$ and $Y$ are ENR spaces, then the above estimate holds for 
singular cohomology with any ring coefficients as well. 
\end{theorem}
\begin{proof}
The first claim follows from the preceding discussion. For the second claim we use the fact 
that on ENR spaces \v Cech cohomology is naturally isomorphic to the singular cohomology 
(cf. \cite[Proposition VIII, 6.12]{Dold}). Note the interesting conclusion that the statement 
about the triviality of cohomology products holds in spite of the fact that the relative cohomology 
product in singular cohomology is in general not defined for non-excisive pairs. 
\end{proof}

Although the theorem is formulated in general terms, we will mostly consider the cases when 
$\cH^*(X\times Y)\cong \cH^*(X)\otimes \cH^*(Y)$. Then the action of $(1,f)^*$ on decomposable tensors is given as
$$u\in \cH^*(X), v\in \cH^*(Y),\;\;\; (1,f)^*(u\otimes v)=u\cdot f^*(v)\in \cH^*(X).$$
Normally we do not attempt to compute the entire kernel of the homomorphism $(1,f)^*$ but we rather look for specific 
elements in the kernel and try to find long non-trivial products. A common source of elements in $\Ker (1,f)^*$ are 
classes of the form $f^*(v)\otimes 1-1\otimes v$ for $v\in \cH^*(Y)$.

\subsection{Summary of main estimates}
\label{sec:summary}
For the convenience of the reader, we summarize in one place the main
estimates for the topological complexity of an arbitrary map. 

Let $f\colon X\to Y$ be any map. 

\begin{enumerate}
\item[$\bullet$] $\max\{\cat(Y),\sec(f)\}\le \TC(f)\le \cat(X)\cdot(\sec(f)+1)-1$
\item[$\bullet$] $f$ simplicial $\Rightarrow$ $\TC(f)\le \cat(X)\cdot(\dim(Y)+1)-1$
\item[$\bullet$] $f$ admits a section $\Rightarrow$ $\TC(Y)\le \TC(f)\le\TC(X)$
\item[$\bullet$] $f$ fibration $\Rightarrow$ $\TC(f)\le\TC(Y)$
\item[$\bullet$] $\TC(f)$ is FHE invariant
\item[$\bullet$] $v\colon \widehat X\to X$ deformation retraction $\Rightarrow$
$\TC(f)=\TC(fv)$
\item[$\bullet$] $f\simeq g$, $g$ fibration $\Rightarrow$ $\TC(g)\le \TC(f)$
\item[$\bullet$] $\bar f$ fibrational substitute for $f$ $\Rightarrow$ $\TC(\bar f)\le \TC(f)$ 
\item[$\bullet$] $\TC(f)\ge\nil(\Ker(1,f)^*\colon \check H^*(X\times Y)\to \check H^*(X))$
\end{enumerate}

For completeness we state without proof the following estimates (see \cite[Proposition 5.5 and Theorem 6.1]{Pav:Topologist's view}).

\begin{enumerate}
\item[$\bullet$] Product formula:  for $f\colon X\to Y$ and $f'\colon X'\to Y'$ we have 
$$\max\{\TC(f),\TC(f')\}\le \TC(f\times f') \le \TC(f)+\TC(f')-1.$$
\item[$\bullet$] 
For every partition $X\times Y=G_1\sqcup\ldots\sqcup G_n$ into disjoint subsets admitting
a partial section to $\pi_f$ there exists a point $(x,y)\in X\times Y$ such that every neighbourhood of it 
intersects at least $\TC(f)$ different domains $G_i$. 
\end{enumerate}

\section{Topological complexity of a fibration}
\label{sec:Topological complexity of a fibration}

As seen in the previous sections, several results about topological complexity depend on the assumption that some 
of the maps involved are fibrations. We will now explore this situation more thoroughly. Furthermore, as explained in 
Section \ref{sec:Definition}, the invariants $\sec$ and $\sec_\mathrm{op}$ coincide 
for fibrations whose base is an ANR. 
We will thus reiterate our standing assumption that $X$ and $Y$ are metric ANR's.

\begin{lemma}
The map $f\colon X\to Y$ is a fibration if, and only if, the induced map $\pi_f\colon X^I\to X\times Y$ is a fibration.
\end{lemma}
\begin{proof}
If $f$ is a fibration, then $1\times f\colon X\times X\to X\times Y$ is also a
fibration, thus $\pi_f$ can be written as a composition of two fibrations.
$$\xymatrix{
X^I\ar[d]_\pi \ar[dr]^{\pi_f} \\
X\times X\ar[r]_-{1\times f} & X\times Y}
$$
Conversely, assume $\pi_f$ is a fibration and consider arbitrary maps
$h$ and $H$ for which the following diagram commutes
$$\xymatrix{
A\ar[r]^h \ar@{^(->}[d] & X \ar[d]^f\\
A\times I \ar[r]_-H & Y}
$$
It gives rise to the following commutative diagram
$$\xymatrix{
A\ar[r]^k \ar@{^(->}[d] & X^I \ar[d]^{\pi_f}\\
A\times I \ar[r]_-K \ar@{-->}[ur]^{\widetilde K}& X\times Y}
$$
where $k(a)=\mathrm{const}_a$, $K(a,t)=(h(a), H(a,t))$, and $\widetilde K$
exists, because $\pi_f$ is a fibration. Then the map $\widetilde H\colon A\times I\to X$, 
defined by $\widetilde{H}(a,t):=\widetilde{K}(a,t)(1)$ is a suitable
lifting of $H$ in the first diagram, which proves that $f$ is a fibration.
\end{proof}

Since a homotopy section of a fibration can be always replaced by a strict section, we 
immediately obtain the following description of the topological complexity of a fibration.

\begin{corollary}
If $f\colon X\to Y$ is a fibration, then
$$\TC(f):=\mathrm{secat}(\pi_f\colon X^I\to X\times Y).$$
\end{corollary}

It is often useful to restate the definition of $\TC(f)$ in more geometric terms, based on the following 
characterization (cf. \cite[Lemma 4.2.1 and Proposition 4.2.4]{Farber:ITR} for analogous description of $\TC(X)$).

\begin{proposition}
\label{prop:equivalent A}
Let $f\colon X\to Y$ be a fibration, and let $A\subseteq X\times Y$. Then the following statements are equivalent:
\begin{enumerate}
\item $A$ admits a partial section $s\colon A\to X^I$ to the projection $\pi_f$;
\item The maps $f\circ\mathrm{pr}_1,\mathrm{pr}_2\colon A\to Y$ are homotopic;
\item $A$ can be deformed in $X\times Y$ to the graph $\Gamma_f$ of the map $f$.
\end{enumerate}
\end{proposition}
\begin{proof}
Let us denote by $\widehat s\colon A\times I\to X$ the adjoint of the partial section $s\colon A\to X^I$. Then 
$f\circ\widehat s\colon A\times I\to Y$ is clearly a homotopy between $f\circ\mathrm{pr}_1$ and $\mathrm{pr}_2$.
Conversely, given a homotopy $H\colon A\times I\to Y$ between $f\circ\mathrm{pr}_1$ and $\mathrm{pr}_2$ one can 
use the fibration property to lift it to a homotopy $\widetilde{H}\colon A\times I\to X$, starting at $\widetilde{H}_0=\mathrm{pr}_1$.
Then the adjoint of $\widetilde{H}$ is a partial section to $\pi_f$ over $A$.

In a similar vein, if $s\colon A\to X^I$ is a partial section to $\pi_f$, then we may define a homotopy 
$H\colon A\times I\to X\times Y$ as $H(a,t):=\big(s(a)(\frac{t}{2}),f\big(s(a)(1-\frac{t}{2})\big)\big)$ 
and check that it defines a deformation of $A$ to $\Gamma_f$. On the other hand, let $H\colon A\times I\to X\times Y$ be
a deformation of $A$ to $\Gamma_f$. Then we define a homotopy $K\colon A\times I\to Y$ by $K(a,t):=\mathrm{pr}_2(H(a,1-t))$ and
lift it along the fibration $f$ to a homotopy $\widetilde K\colon A\times I\to X$ with $\widetilde K_0=\mathrm{pr}_1\colon A\to X$. 
It is easy to check that the adjoint of $\widetilde{K}$ is a partial section to $\pi_f$ over $A$.
\end{proof}

\begin{corollary}
If $f\colon X\to Y$ is a fibration, then $\TC(f)$ equals the minimal number of elements of a covering of $X\times Y$ by open sets
that can be deformed in $X\times Y$ to the graph of $f$.
\end{corollary}

As we mentioned in Remark \ref{rem:TC def}, for a large class of spaces $X$ one can compute $\cat(X)$ and $\TC(X)$ by taking
arbitrary subspaces of $X$ or $X\times X$ as domains of partial sections. We are going to show that an analogous result holds 
for the topological complexity of a fibration. 

\begin{lemma}
\label{lem:Srinivasan}
Let $f,g\colon X\to Y$ be continuous maps between compact metric ANR spaces, and let $A$ be an arbitrary subset of $X$.
If $f|_A\simeq g|_A$, then  there exists an open neighbourhood $U\subseteq X$ of $A$ such that $f|_U\simeq g|_U$.
\end{lemma}
\begin{proof}
For simplicity we will use the same notation $d$ for the metrics in $X$ and $Y$ and also for the induced supremum metric on 
the space of path $Y^I$. 

We will need the following standard properties of maps into metric ANR spaces:

$\bullet$ For every compact metric ANR space $E$ there exist an $\varepsilon>0$, such that every two maps $f,g\colon X\to E$ 
that are $\varepsilon$-close (i.e. $d(f(x),g(x))<\varepsilon$ for all $x\in X$) are homotopic (cf. \cite[Theorem 2.4]{Srinivasan}). 

$\bullet$ (Walsh lemma) Assume that $X$ and $E$ are separable metric spaces, and furthermore, that $E$ is an ANR. Let $h\colon A\to E$
be a continous map defined on an arbitrary subset $A\subseteq X$. Then, up to a small homotopy, $h$ can be extended to an open neighbourhood
of $A$. More precisely, for every $\varepsilon,\delta>0$ there exists an open 
subset $U\subseteq X$ containing $A$ and a map $\overline h\colon U\to E$, satisfying the following conditions:\\
\hspace*{2mm}(1) for every $u\in U$ there exists $a\in A$ such that $d(u,a)<\delta$ and $d(\overline h(u),h(a))<\varepsilon$;\\
\hspace*{2mm}(2) $\overline h|_A\simeq h$\\
(cf. \cite[Theorem 2.3]{Srinivasan} and the comments at the end of the proof therein).

Returning to the proof of our statement, let $\varepsilon>0$ be such that any two $\varepsilon$-close maps $Y$ are homotopic. Since $X$ 
is compact, $f$ and $g$ are uniformly continuous, so there exists $\delta>0$ such that $d(x,x')<\delta$ imply
$d(f(x),f(x'))<\frac{\varepsilon}{2}$ and $d(g(x),g(x'))<\frac{\varepsilon}{2}$.
The homotopy $H\colon A\times I\to Y$ between $f$ and $g$ corresponds by adjunction to a map $\widehat H\colon A\to Y^I$. 
It is well-known that if $Y$ is a compact metric ANR then $Y^I$ is a metric ANR. 
Thus we may apply the Walsh lemma to obtain an open neighbourhood $U$ of $A$ and a map $G\colon U\to Y^I$, such that 
for every $u\in U$ there exists $a_u\in A$ satisfying $d(u,a_u)<\delta$ and $d(G(u),\widehat H(a_u))<\frac{\varepsilon}{2}$
(i.e. $d(G(u)(t),\widehat H(a_u)(t))<\frac{\varepsilon}{2}$ for all $t\in I$). Define $G_0,G_1\colon U\to Y$ as $G_0(u):=G(u)(0)$ and 
$G_1(u):=G(u)(1)$.
Then for every $u\in U$ we have the triangle inequality (note that $H(a_u)(0)=f(a_u)$) 
$$d(G_0(u),f(u))\le d(G_0(u),\widehat H(a_u)(0))+d(f(a_u),f(u)) <\frac{\varepsilon}{2}+\frac{\varepsilon}{2}=\varepsilon.$$
As a consequence, $G_0$ and $f|_U$ are homotopic, and similarly for $G_1$ and $g|_U$. Since $G_0$ and $G_1$ are homotopic by construction, 
we conclude that $f|_U\simeq g|_U$ as claimed. 
\end{proof}

\begin{theorem}
Let $f\colon X\to Y$ be a fibration between compact metric ANR spaces $X$ and $Y$. Then $\TC(f)$ is equal to the minimal integer $n$ 
for which there exists a cover 
$$X\times Y=A_1\cup\ldots\cup A_n$$ 
such that each $A_i$ admits a continuous partial section to $\pi_f$.
\end{theorem}
\begin{proof}
It is clearly sufficient to show that each $A_i$ is contained in some open set that admits a partial section to $\pi_f$.

If $A_i$ admits a partial section to $\pi_f$ then the maps $f\circ\mathrm{pr}_1,\mathrm{pr}_2\colon A_i\to Y$ are homotopic
by Proposition \ref{prop:equivalent A}. Observe that $f\circ\mathrm{pr}_1$ and $\mathrm{pr}_2$
are defined on entire $X\times Y$. We may thus apply Lemma \ref{lem:Srinivasan}
to obtain an open neighbourhood $U_i\subseteq X\times Y$ of $A_i$, such that the maps $f\circ\mathrm{pr}_1,\mathrm{pr}_2\colon U_i\to Y$ 
are homotopic. Again by Proposition \ref{prop:equivalent A} it follows that $U_i$ admits a continuous partial section to $\pi_f$.
\end{proof}

Most estimates of $\TC(f)$ can be considerably strengthened if we assume that $f$ is a fibration. 

\begin{proposition}
\label{prop:TC(Y)&catXY}
If $f$ is a fibration then
$$\cat(Y)\le\TC(f)\le\min\{\TC(Y),\cat(X\times Y)\}.$$
In particular, $\TC(f)=1$ if, and only if $Y$ is contractible.
\end{proposition}
\begin{proof}
By Proposition \ref{prop:TC and cat} $\TC(f)\ge\cat(Y)$, and by Corollary \ref{cor:fib & sec} $\TC(f)\le\TC(Y)$. Moreover, since $\pi_f$ is a fibration,
there exists a partial section to $\pi_f$ over every categorical subset
of $X\times Y$. As a consequence $\TC(f)\le\cat(X\times Y)$.
\end{proof}

If $Y$ is a topological group (or more generally, for an H-group), then the complexity of $Y$ coincide with its category, so we obtain the following result:

\begin{corollary}
\label{cor:contrX}
Let $f\colon X\to Y$ be a fibration, and assume that $X$ is contractible or that  $Y$ is an H-group. Then $\TC(f)=\cat(Y)$.
\end{corollary}

The following theorem allows a more detailed description of $\TC(f)$.

\begin{theorem}
\begin{enumerate}
\item[\textbf{a)}]
If $f\colon X\to Y$ is a fibration, then the fibration $\pi_f\colon X^I\to X\times Y$ is fibre-homotopy equivalent
to the projection $q\colon X\sqcap Y^I\to X\times Y$ given by $q(x,\alpha):=(x,\alpha(1))$.
\item[\textbf{b)}]
Furthermore, the following diagram is a pull-back
$$\xymatrix{
X\sqcap Y^I \ar[r] \ar[d]\pullbackcorner  & Y^I \ar[d]^\pi\\
X\times Y \ar[r]_{f\times 1} &Y\times Y
}$$
so in particular $q\colon X\sqcap Y^I\to X\times Y$ is a fibration with fibre $\Omega Y$.
\end{enumerate}
As a consequence, if $f\colon X\to Y$ is a fibration, then $\TC(f)$ equals the sectional category of the fibration
$q\colon X\sqcap Y^I\to X\times Y$.
\end{theorem}
\begin{proof}
\begin{enumerate}
\item[\textbf{a)}]
Recall that $f\colon X\to Y$ is a fibration if, and only if, there exists a lifting function $\Gamma_f\colon X\sqcap PY\to X^I$,
which is, by definition, a section to the natural projection $p\colon X^I\to X\sqcap Y^I$, given by 
$p(\alpha)=(\alpha(1), f\circ\alpha)$.
This may be restated by saying that $\Gamma_f$ and $p$ are fibrewise maps over $X\times Y$ as in the following 
commutative diagram (where $q(x,\alpha)=(x,\alpha(1)))$.
$$\xymatrix{
X^I \ar@<0.5ex>[rr]^p\ar[dr]_{\pi_f} & & {X\sqcap Y^I} \ar@<0.5ex>[ll]^{\Gamma_f}\ar[dl]^q \\
& X\times Y
}$$
Since $p\circ \Gamma_f=1_{X\sqcap Y^I}$ and $\Gamma_f\circ p$ is fibre-homotopic to $1_{X^I}$ we conclude that $\pi_f$
and $p$ are fibre-homotopy equivalent. 
\item[\textbf{b)}]
The second statement follows from the following computation
\begin{equation*}
\begin{split}
(X\times Y)\sqcap Y^I & =  \{(x,y,\alpha)\in X\times Y\times Y^I\mid f(x)=\alpha(0), y=\alpha(1)\}\\
  & =  \{(x,\alpha)\in X\times Y^I\mid f(x)=\alpha(0)\}= X\sqcap Y^I
\end{split}
\end{equation*}
Being a pull-back of the path-fibration $\pi\colon Y^I\to Y\times Y$, the map $q$ is also a fibration, with the same fibre as 
$\pi$, which is the loop space $\Omega Y$.
\end{enumerate}
We conclude the proof by observing that fibre-homotopy equivalent fibrations have the same sectional category.
\end{proof}

It may be worth noting that we have actually proved that if $f\colon X\to Y$ is a fibration, then the diagram
$$\xymatrix{
X^I \ar[r]^{f\circ -} \ar[d]_{\pi_f}\pullbackcorner  & Y^I \ar[d]^{\pi_Y}\\
X\times Y \ar[r]_{f\times 1} &Y\times Y
}$$
is a homotopy pull-back. Since the pull-back operation cannot increase sectional
category, we immediately deduce $\TC(f)=\secat(\pi_f)\le\secat(\pi_Y)=\TC(Y)$.
On the other hand the sectional category of a fibration is smaller or equal to the category of the base, therefore $\TC(f)\le\cat(X\times Y)$. We have thus
obtained an alternative proof of Proposition \ref{prop:TC(Y)&catXY}.

\begin{example}
\label{ex:basicTC}
\begin{enumerate}
\item $\TC(X\to\{y\})=1$, by Corollary \ref{cor:fib & sec}(b).
\item $\TC(\ev_1\colon PX\to X)=\cat(X)$, by Corollary \ref{cor:contrX}. 
\item $\TC(\ev_1\colon X^I\to X)=\TC(X)$, by Corollary 	\ref{cor:fib & sec}(b).
\item $\TC({\rm pr}_X\colon X\times F\to X)=\TC(X)$ by Corollary \ref{cor:fib & sec}(b). This example shows
that the complexity of a map $f\colon X\to Y$ can be much smaller than $\cat(X\times Y)$. 
\end{enumerate}
\end{example}

One very useful estimate of the topological complexity of a space is the 'dimension divided by connectivity'
bound (see \cite{Farber:IRM}): 
if $X$ is $\dim(X)$-dimensional and $\conn(X)$-connected, then 
$$\TC(X)\le\left\lfloor\frac{2\dim(X)}{\conn(X)+1}\right\rfloor+1,$$
(where $\lfloor r \rfloor$ stands for the value of $r$ rounded down to the closest integer).
The result is proved by obstruction theory applied to the Schwarz's \cite{Schwarz}
characterization of the sectional category. One could follow the same approach to estimate the sectional
category of the fibration $q\colon X\sqcap PY\to X\times Y$ with fibre $\Omega Y$, but it turns out that an even better estimate can be obtained by combining Proposition \ref{prop:TC(Y)&catXY} 
with the dimension divided connectivity estimate for the category (\cite{CLOT}...).

\begin{corollary}
\label{cor:dim-conn}
If $f\colon X\to Y$ is a fibration then
$$\TC(f)\le\min\left\{\left\lfloor\frac{\dim(X)}{\conn(X)+1}\right\rfloor,\left\lfloor\frac{\dim(Y)}{\conn(Y)+1}\right\rfloor\right\}+\left\lfloor\frac{\dim(Y)}{\conn(Y)+1}\right\rfloor+1,$$
\end{corollary}
\begin{proof}
We may restate Proposition \ref{prop:TC(Y)&catXY} as 
$$\TC(f)\le\min\{\cat(X\times Y),\cat(Y\times Y)\}.$$
Then the combination of the bound for the category of a product 
$$\cat(X\times Y)\le\cat(X)+\cat(Y)-1,$$ with the
'dimension divided connectivity' bound for the category \cite{CLOT} 
$$\TC(f)\le\left\lfloor\frac{\dim(X)}{\conn(X)+1}\right\rfloor$$
yields the stated result.
\end{proof}

\begin{example}
\begin{enumerate}
\item Consider the covering map $p\colon S^n\to\RR P^n$: since dimension-to-connectivity ratio is smaller 
for the sphere than for the projective plane, Corollary \ref{cor:dim-conn} yields $\TC(p)\le 1+n+1=n+2$.
In comparison, $\TC(\RR P^n)$ is usually much bigger and closer to $2n$ (cf. \cite{Farber:ITR}).
\item Similarly, for the standard quotient map $q\colon S^{2n+1}\to\CC P^n$ we obtain the estimate
$\TC(q)\le n+2$, which is much smaller that $\TC(\CC P^n)=2n+2$.
\item For a fibration over a sphere $f\colon X\to S^n$ we obtain $2=\cat(S^n)\le\TC(f)\le 3$. Observe that if $n$ is
odd, we have $\TC(f)=2$ by Corollary \ref{cor:fib & sec}, and the difference is caused by the fact that for odd-dimensional
sphere the dimension-to-connectivity estimate is not sharp. 
\end{enumerate}
\end{example}

Let us illustrate the use of the cohomological estimate in the computation of the topological complexity of a map. 

There are many fibrations for which $f^*\colon H^*(Y)\to H^*(X)$ is trivial (examples include $p\colon S^n\to \RR P^n$,
$q\colon S^{2n+1}\to \CC P^n$, Hopf fibrations,...). In that case non-trivial elements
in $\Ker (1,f)^*$ must be contained in $\oplus_{j>0} H^i(X)\otimes H^j(Y)$. It follows that every $k$-fold product
in $\Ker (1,f)^*$ 'contains'a $k$-fold product in $H^*(Y)$, therefore
$$	\nil(\Ker (1,f)^*)\le\nil(H^*(Y))\le\cat(Y),$$
so if $f^*=0$ the cohomology estimate does not improve the estimate $\TC(f)\ge\cat(Y)$.

\begin{example}
Let $f\colon SO(n)\to S^{n-1}$ be the standard fibration obtained by projecting each orthogonal matrix to its last column. 
If $n$ is even, then 
$$2=\cat(S^{n-1})\le \TC(f)\le \TC(S^{n-1})=2,$$ 
hence $\TC(f)=2$. However, if $n$ is odd, then $2\le\TC(f)\le 3$, and we are going to use the cohomology estimate
to show that the actual value is 3. In fact, it is well known that the image $f^*(u)$ of a generator $u\in H^{n-1}(S^{n-1})$ 
is a non-trivial element of $H^{n-1}(SO(n)$ because it reduces to one of the standard generators of $H*{n-1}(SO(n);\ZZ/2)$. 
Therefore $f^*(u)\otimes 1-1\otimes u\in \Ker(1,f)^*$ and 
$$(f^*(u)\otimes 1-1\otimes u)^2=-2 f^*(u)\otimes u\ne 0.$$
We conclude that $\TC(f)=3$.
\end{example}

The above example is an instance of a general situation when $f^*\colon H^*(Y)\to H^*(X)$ is injective. If we apply 
a cohomology functor $H^*$ to the following commutative diagram
$$\xymatrix{
X \ar[r]^f \ar[d]_{(1,f)} & Y \ar[d]^\Delta\\
X\times Y \ar[r]_-{f\times 1} & Y\times Y
}$$
and assume that $H^*$ has field coefficients or that $H^*(Y)$ is free, and that $f^*$ is injective. Then we obtain the diagram
$$\xymatrix{
H^*(X) &  & H^*(Y) \ar@{>->}[ll]_{f^*} \\
H^*(X\times Y) \ar[u]^{(1,f)^*} & & H^*(Y\times Y) \ar[u]_{\Delta_Y^*} \ar[ll]_-{(f\times 1)^*}\\
H^*(X)\otimes H^*(Y) \ar[u]^\cong & & H^*(Y)\otimes H^*(Y) \ar[u]_\cong \ar@{>->}[ll]_{f^*\otimes 1}
}$$
Observe that the $f^*\otimes 1$ is injective because we assumed that either $H^*$ has field coefficients or that
$H^*(Y)$ is free, and tensoring with a free module preserves injectivity. 
The commutativity of the diagram implies that
we can identify $\Ker\Delta_Y^*$ with a subideal of $\Ker(1,f)^*$, so we have proved the following result:

\begin{theorem}
\label{thm:mono coh}
Let $f\colon X\to Y$ be any map and assume that we consider a cohomology with field coefficients 
or that $H^*(Y)$ is free.
If $f^*\colon H^*(Y)\to H^*(X)$ is injective, then $TC(f)\ge \nil(\Ker\Delta_Y^*)$.

If, in addition, $f$ a fibration,  then 
$\nil(\Ker\Delta_Y^*)\le \TC(f)\le \TC(Y).$
\end{theorem}

Note that the nilpotency of $\Ker\Delta_Y^*$ was introduced by Farber \cite{Farber:TCMP} 
(under the name of 'zero divisors cup length') as the basic lower bound for the topological complexity. 
In many cases (in fact, in almost all cases where the exact value of $\TC(Y)$ is known)
$\nil(\Ker\Delta^*_Y)$ is either equal to $\TC(Y)$ or to $\TC(Y)-1$, so the above estimate is 
a very useful tool for computations. 

An important class of maps to which the above Theorem applies are fibre bundles whose fibres
are totally non-homologous to zero. Recall that the fibre $F$ of a fibration $f\colon X\to Y$
is said to be \emph{totally non-homologous to zero} with respect to a field $R$ if the homomorphism
$H^*(X;R)\to H^*(F;R)$ induced by the inclusion of the fibre is surjective. If that case 
the Serre spectral sequence for $f$ collapses at the $E_2$-term, which in turn implies that 
$f^*\colon H^*(Y;R)\to H^*(X;R)$ is injective. 

\begin{corollary}
If $f\colon X\to Y$ is a fibration whose fibre is totally non-homologous to zero with respect to 
a field $R$, and if $\TC(Y)=\nil(\Ker\Delta_Y^*)$ (cohomology with coefficients in $R$), then 
$\TC(f)=\TC(Y)$.
\end{corollary}

Let $X$ be a pointed CW-complex (we omit the base-point from the notation), and let ${\rm Cov}(X)$ denote 
the set of (equivalence classes) of base-point preserving covering projections over $X$. It is well-known 
that there is a bijection between $\Cov(X)$ and the lattice of subgroups of the fundamental group $\pi_1(X)$. 
To every $G\le\pi_1(X)$ there corresponds a unique $p_G\colon\wX_G\to X$ such that $\im (p_G)_\sharp=G$. 
In particular, $p_{\pi_1(X)}=\id_X$ and $p_{\{1\}}$ is the universal covering projection over $X$.

If $G,G'$ are subgroups of $\pi_1(X)$, then the lifting criterion for covering spaces implies that 
$G\le G'$ if, and only if, there exists a map
$v\colon \wX_G\to \wX_{G'}$ such that the following diagram commutes
$$\xymatrix{
\widetilde X_G \ar@{-->}[rr]^v \ar[dr]_{p_G} & & \widetilde X_{G'}\ar[dl]^{p_{G'}}\\
& X
}$$ 
Moreover, when such $v$ exists it is unique and it is itself a covering projection. 
Therefore, if $G\le G'\leq\pi_1(X)$, then there is a fibration $v$ such that $p_{G'}v=p_G$, and 
Theorem \ref{thm:precomp fibration} implies that $\TC(p_G)\le \TC(p_{G'})$. We have thus proved

\begin{theorem}
\label{thm:TCcovering}
The topological complexity of covering projections determines an increasing map 
from the lattice of subgroups of $\pi_1(X)$ to $\NN\cup\{\infty\}$.
Its minimal value is the topological complexity of the universal covering projection
and its maximal value is $\TC(X)$.
\end{theorem}

Observe that for an arbitrary covering projection $p\colon \wX\to X$ Propositon \ref{prop:TC(Y)&catXY} implies 
the estimate 
$\cat(X)\le\TC(p)\le\cat(X\times \wX),$
which is often easier to compute.

Let us now study more closely covering projections over Eilenberg-MacLane spaces. The homotopy type of 
an Eilenberg-MacLane space $K(G,1)$ is uniquely determined by the group $G$. 
As a consequence both $\cat(K(G,1))$ and $\TC(K(G,1))$ are in fact invariants of $G$ and are often 
denoted as $\cat(G)$ and $\TC(G)$, respectively. Every covering projections over $K(G,1)$ corresponds to 
a subgroup $H\le G$ and its total space is in fact an Eilenberg-MacLane space of type $K(H,1)$. 
Since the universal covering space of $K(G,1)$ is contractible we have $\TC(p_{\{1\}})=\cat(G)$ by \ref{cor:contrX}.
Theorem \ref{thm:TCcovering} then yields a general estimate 
$$\cat(G)\le \TC\big(p\colon K(H,1)\to K(G,1)\big)\le\TC(G).$$
Note that if $G$ is abelian then $K(G,1)$ is an $H$-group and Corollary \ref{cor:contrX} implies that $\TC(p)=\cat(G)$
for every covering projection $p$ with base $K(G,1)$. 

We also give two non-commutative examples. Let $p\colon \wX\to S^1\vee S^1$ be the universal covering of the wedge 
of two circles. Since $\wX$ is contractible, we get $\TC(p)=2$, while $\TC(S^1\vee S^1)=3$.
Similarly, let $S$ be a closed surface different from the sphere or projective plane, and let $p\colon \widetilde S \to S$
be its universal covering. Then $\widetilde S$ is contractible, therefore $\TC(p)=\cat(S)=3$ while $\TC(S)=5$. 

\begin{remark}
Eilenberg and Ganea \cite{Eilenberg-Ganea} showed that $\cat(G)$ can be expressed in a completely algebraic manner:
they proved that $\cat(G)=\cat(K(G,1)=\cd(G)+1$, where $\mathrm{cd}$ denotes the \emph{cohomological dimension} of $G$. 

At this moment there is no completely algebraic way to compute $\TC(G)$. We have the general estimate
$$\cd(G)+1=\cat(K(G,1))\le\TC(G)\le \cat(K(G,1)\times K(G,1))=\cd(G\times G)+1.$$
Rudyak \cite{Rudyak} proved that for a suitable choice of group $G$ the value of $\TC(G)$ can be any number between 
$\cd(G)+1$ and $\cd(G\times G)+1$. On the other hand it has been recently proved by Farber and Mescher \cite{Farber-Mescher} 
that for a large class of groups (including all hyperbolic groups) $\TC(G)$ is either $\cd(G\times G)$ or $\cd(G\times G)+1$.
\end{remark}

We conclude with a partial result about finite-sheeted covering projections.

\begin{theorem}
Assume that the topological complexity of $X$ equals the rational cohomological lower bound $\TC(X)=\nil(\ker H^*(\Delta;\QQ))$.
Then $\TC(p)=\TC(X)$ for every finite-sheeted covering projection $p\colon \wX\to X$.
\end{theorem}
\begin{proof}
Recall that finite-sheeted covering projections induce monomorphisms in rational cohomology (see \cite[Proposition 3G.1]{Hatcher}). 
Then the claim follows directly from Theorem \ref{thm:mono coh}.
\end{proof}

For instance, the topological complexity of every finite-sheeted covering over an orientable surface $P$
of genus bigger then 1 is equal to $\TC(P)=5$, while the topological complexity of its universal cover is
equal to $\cat(P)=3$. We do not know whether there are covering projections to $P$ whose topological
complexity is 4. On the other hand we suspect that $\TC(p)=\TC(X)$ for every finite sheeted covering 
projection $p$ with base $X$.

\section*{Acknowledgements}

We are grateful to Nick Callor for helpful discussions on certain aspects of the article, 
in particular for his suggestion to base the definition of topological complexity on open
filtrations. We are also indebted with Cesar Zapata and Michael Farber who helped us 
identify and correct errors that appeared in a previous version of the paper (Theorems 
\ref{thm:upper bound} and \ref{thm:coho bound}).

\end{document}